\documentclass[12pt]{amsart}
\usepackage{amsmath, amsfonts, yfonts, amssymb, amscd, epsfig}

\theoremstyle{plain}
\newtheorem{thm}{Theorem}[section]
\newtheorem{lem}[thm]{Lemma}

\newcommand{\R}{{\mathcal R}}

\def \PO {{\mathbb{C}{P}}^{1}}
\def \R {\mathbb{R}}

\def \Sig{\Sigma}

\def \x {\times}
\def \- {\setminus}

\def\dfn#1{{\em #1}}

\title[Stein fillings]{A note on Stein fillings \\ of contact manifolds}

\author{Anar Akhmedov}

\author{John B. Etnyre}

\author{Thomas E. Mark}

\author{Ivan Smith}

\begin{document}

\begin{abstract}
In this note we construct infinitely many distinct simply connected Stein fillings of a certain infinite family of contact 3--manifolds. 
\end{abstract}

\maketitle

\section{Introduction}

Given a contact 3--manifold $(M,\xi)$ it is useful to understand the types of fillings it has. In
particular, Stein and strong fillings of $(M,\xi)$ can be used in symplectic cut-and-paste
procedures. 
Recall $X$ is a \dfn{Stein filling} of $(M,\xi)$ if $X$
is the sub-level set of a plurisubharmonic function on a Stein surface and the contact structure
induced on $\partial X$ by complex tangencies is contactomorphic to $(M,\xi).$ Many early results concerning Stein fillings pointed to a finiteness result.
For example, Eliashberg \cite{Eliashberg90b} showed the tight contact structure on $S^3$ has a unique Stein filling. McDuff \cite{McDuff90} and Lisca \cite{Lisca04} have shown that all tight contact structures on Lens spaces have a finite number of Stein fillings. Several other uniqueness results are known \cite{OhtaOno03, OhtaOno05, Schoenenberger05}. However, recently Ozbagci and Stipsicz \cite{OzbagciStipsicz04a} and independently Smith \cite{Smith01} have shown that certain contact structures have infinitely many Stein fillings. The fillings in these examples had non-trivial fundamental group (they were distinguished by torsion in the first homology group). As various constructions in 4--manifold topology are simpler when the fundamental group is not controlled it was natural to ask if there is a similar non-finiteness result for simply connected Stein fillings. Our main result shows there is.
\begin{thm}
There is a sequence of distinct contact manifolds $(M_i, \xi_i),$ each of which has an infinite number of homeomorphic but non-diffeomorphic, simply connected Stein fillings. 
\end{thm}
These examples are a simple consequence of Fintushel and Stern's work \cite{FintushelStern04} but fill a gap in the literature concerning Stein fillings of contact structures. We point out that in combination with Ozbagci and Stipsicz's work \cite{OzbagciStipsicz04a} these manifolds $(M_i,\xi_i)$ have infinitely many homeomorphism types of  Stein fillings and some homeomorphism types of fillings have infinitely many diffeomorphism types of Stein fillings. One issue that is tantalizingly left open is whether or not there is a contact 3--manifold whose Stein fillings realize infinitely many values for their Euler characteristics $\chi,$ or realize infinitely many pairs of values for $(\chi,\sigma),$ where $\sigma$ is the signature. Given any natural number $n$ one may easily check that the construction below will provide a contact manifold $(M,\xi)$ realizing $n$ distinct $(\chi,\sigma)$ pairs for its Stein fillings.

Acknowledgments: The authors thank Selman Akbulut, Ciprian Manolescu, Peter Ozsv\'ath and Ron Stern for useful discussions in the preparation of this note. AA was partially supported by NSF grant DMS-0244663. JE was partially supported by NSF grants DMS-0707509 and 0244663.
\section {Knot Sugery and Lefschetz Fibrations on $E(n)_K$}

\subsection{Lefschetz fibrations on $E(n)$}
Recall that the elliptic surface $E(n)$ can be discribed as the desingularization of a double branched cover. Specifically, consider the double branched cover of $\PO\times\PO$ whose branch set $B_{2,n}$ is the union of four disjoint copies of $\PO\x \{{\rm{pt}}\}$ and  $2n$ disjoint copies of $\{{\rm{pt}}\}\x \PO$. This branch cover has $8n$ singular points corresponding to the intersections of horizontal and vertical lines in the branch set $B_{2,n}$. Each singular point is easily seen to be a cone on $\R P^3.$ After desingularizing the above space, that is removing a neighborhood of the singular points and replacing them with unit cotangent bundles over $S^2$ (which is of course a $D^2$ bundle over $S^2$ with normal Euler number $-2$),  one obtains $E(n)$. 

The horizontal and vertical fibrations of $\PO\times\PO$ pull back to give fibrations of $E(n)$ over $\PO$. A generic fiber of the vertical fibration is the double cover of $S^2$, branched over $4$ points. Thus a generic fiber will be a torus and the fibration is an elliptic fibration on $E(n)$. The generic fiber of the horizontal fibration is the double cover of $S^2$, branched over $2n$ points, which is a genus $n-1$ fibration on $E(n)$. This fibration has $4$ singular fibers which correspond to preimages of the four copies of $S^{2}\times {pt}$'s in the branch set together with the spheres of self-intersection $-2$ coming from the desingularization.  While this fibration is not a Lefschetz fibration it was shown in \cite{FintushelStern04} that it can be slightly deformed near the singular fibers to become a Lefschetz fibration.  Notice that the generic fiber of the horizontal fibration $\Sig_{n-1}$ intersects a generic fiber $F$ of the elliptic fibration in two points.  

\subsection{Lefschetz Fibrations on $E(n)_K$}
Let $K$ be a fibered knot of genus $g$, and $F$ be  a generic torus fiber $F$ of $E(n)$. Then the knot surgered elliptic surface is a manifold 
\[
E(n)_K = (E(n)\- (F\x D^2)) \cup (S^1\x (S^3\- N(K)),
\] 
where each normal $2$-disk to $F$ is replaced by a fiber of the fibration of $S^3\- N(K)$ over $S^1$. Since $F$ intersects each generic horizontal fiber twice, we get an induced Lefschetz fibration 
\[ 
h: E(n)_K\to\PO
\]  
with fiber genus  $2g+n-1$.

\begin{lem}\cite{FintushelStern04} 
If $K$ is a fibered knot whose fiber has genus
$g$, then $E(n)_K$  admits a locally holomorphic fibration (over $\PO$) of genus $2g+n-1$ which has exactly four singular fibers. Furthermore, this fibration can be deformed locally to be Lefschetz.
\end{lem} 

We also note the following fact about these Lefschetz fibrations.
\begin{lem}\label{lem:construction}
Let $K$ be a fibered knot of genus $g$. Then for any $n,$ the Lefschetz fibration $E(n)_{K}$ admits a sphere section $S$ of self-intersection $-2$. Moreover, $E(n)_K\setminus S$ is simply connected and still has a sphere section.
\end{lem}
\begin{proof}
Let $S'$ be one of the vertical spheres $\{{\rm{pt}}\}\x \PO$ in $\PO\times\PO$ that is part of the branch locus in the description of $E(n)$ above. In the double branched cover $S'$ lifts to a sphere $S''$ and after desingularizing to obtain $E(n)$ we get a sphere $S$ that intersects each of four of the desingularizing $D^2$ bundles over $S^2$ in a fiber. It is easy to see that $S$ is a section of the genus $n-1$ Lefschetz fibration of $E(n)$ over $\PO$. Moreover, since $S$ intersects one of the desingularizing $S^2$'s in a point, its meridian is null-homotopic in the complement of $S.$ Thus, as $E(n)$ is simply connected, the fundamental group of the complement of $S,$ which is normally generated by its meridian, is trivial.

When normal summing $S^1\times S^3$ to $E(n)$ to construct $E(n)_K,$ the sphere $S$ and the desingularizing $S^2$ are not affected. Thus $S$ is still a section with simply connected complement. Lastly, taking another vertical sphere in the branch locus provides a section that is disjoint from $S.$
\end{proof}

\section{Extensions of diffeomorphisms over plumbings}
Let $X_{g,n}$ be the 4-manifold obtained by plumbing $\Sigma_g\times D^2$ to a $D^2$-bundle over $S^2$ with Euler number $-n.$ The main result of this section is the following extension result.
\begin{lem}\label{lem:extend}
Any orientation preserving diffeomorphism of $\partial X_{g,n}$ extends over $X_{g,n}.$ 
\end{lem}
\begin{proof}
The case when $n=0$ is clear since $X_{g,0}$ is the boundary sum of $2g$ copies of $S^1\times D^3,$ 
if $g>0.$ If $g=0$ then $\partial X_{0,n}=S^3$ and thus the lemma is also clear in this case.

If $n\not=0$ and $g>0$ then $\partial X_{g,n}$ is a Seifert fibered space over $\Sigma_g$ with one singular fiber. Let $\phi$ be a diffeomorphism of $\partial X_{g,n}.$ Let $C=\{c_1,\ldots, c_g\}$ and $C'=\{c_1',\ldots, c_g'\}$ be collections of curves on $\Sigma_g$ such that the curves within each collection are disjoint and $c_i\cap c_j'$ is exactly one point if $i=j$ and empty otherwise. We moreover assume that no curve in $C$ or $C'$ goes through the singular point and $C$ and $C'$ are chosen so that $\Sigma_g\setminus (C\cup C')$ is a disk. Let $T$ and $T'$ be the preimages of $C$ and $C'$ in $\partial X_{g,n}.$ These are collections of incompressible tori. Thus $\phi(C)$ is also a collection of incompressible tori. It is well known in this situation that the collection $\phi(C)$ can be isotoped to be vertical (that is a union of fibers in the Seifert fibration), \cite{Hatcher3mdfNotes, Jaco80}. Extending this isotopy we can assume $\phi(T)$ is a collection of vertical tori. Each component of $T'$ in $\partial X_{g,n}\setminus T$ are incompressible and boundary incompressible annuli. Their images will also be such annuli and thus can be isotoped to be horizontal or vertical. Since a horizontal annulus in $\partial X_{g,n}\setminus T$ would be part of a horizontal torus in $\partial X_{g,n},$ which does not exist, we can further isotop $\phi$ so that $\phi(T')$ and $\phi(T)$ are vertical tori. At this point it is clear that $\phi$ can be further isotoped to preserve the Seifert fibration on $T\cup T'$ and even on a neighborhood of $T\cup T'.$

As the tori in $T\cup T'$ consist of regular fibers we can assume they are contained in the $\Sigma_g\times D^2$ part of $X_{g,n}$ and each such fiber in this part bounds an obvious disk in $X_{g,n}.$ We can now extend $\phi$ over those disks whose boundaries are fibers in a neighborhood of $T\cup T'.$ It remains to extend $\phi$ over the $D^2$-bundle over $S^2.$ The boundary of this bundle is $L(n,1).$ In \cite{Bonahon83} it was shown that, up to isotopy,  there is precisely one non-identity diffeomorphism of $L(n,1).$   Recall $L(n,1)$ is the double branched cover of $S^3$ with branch set the boundary of the unknotted annulus or M\"obius band with $-n$  half twists. The nontrivial diffeomorphism of $L(n,1)$ is simply the covering automorphism of this two fold cover. Since the disk bundle bounded by $L(n,1)$ is the two fold branched cover of $B^4$ branched over the interior of the above mentioned annulus or M\"obius band with its interior pushed into the interior of $B^4$ it is clear that this diffeomorphism extends over the disk bundle.  Thus we may extend $\phi$ over the rest of $X_{n,g}.$
\end{proof}

\section{Stein fillings}
In this section we construct our examples of contact structures with infinitely many distinct, simply connected Stein fillings all of which are homeomorphic. We begin by recalling a well-known relation between Lefschetz fibrations and Stein manifolds.
\begin{lem}\label{lem:comp}
Let $f:X\to \PO$ be a Lefschetz fibration of a 4-manifold that admits a section $S$ of square $-n.$ If $F$ is a generic fiber of the fibration then $X-(F\cup S)$ admits the structure of a Stein manifold. Moreover, the Stein manifold fills the contact 3-manifold $(Y,\xi)$ where $\xi$ is supported by an open book decomposition with page $F\setminus \nu(S\cap F),$ where $\nu(S\cap F)$ is a neighborhood of $S\cap F$ in $F,$ and monodromy $n$-right handed Dehn twists about a curve parallel to the boundary.  
\end{lem}
This well known lemma follows from the handlebody description of Lefschetz fibrations and the relation between open books and contact structures, see \cite{AkbulutOzbagci01, Etnyre06, OzbagciStipsicz04}.

With the above preparation we are now ready to prove our main result.
\begin{thm}
Let $\xi_{g,m}$ be the contact structure on the Seifert fibered space over a surface of genus $g>4$ with one singular fiber of multiplicity $2m$
 supported by the open book with page a surface of genus $g$ with one boundary component and monodromy $2m$ positive Dehn twists about a curve parallel to the boundary. The contact structure is Stein fillable by infinitely many homeomorphic but non-diffeomorphic simply connected Stein manifolds. 
\end{thm}
\begin{proof}
From Lemma~\ref{lem:construction} we know that for a genus $g$ fibered knot $K$ in $S^3,$ $E(n)_K$ has a genus $n-1+2g$ Lefschetz fibration with a section $S$ of square $-2$ and whose complement is simply connected. It is clear that $X_{n-1+2g,-2}$ embeds in $E(n)_K$ as a neighborhood of a fiber and $S.$ From Lemma~\ref{lem:comp} we know that $S_{n,K}=E(n)_K\setminus X_{n-1+2g,-2}$ is a Stein manifold filling the contact structure $\xi_{n-1+2g,1}$ described in the theorem. Moreover, $S_{n,K}$ is seen to be simply connected by the existence of a section of $E(n)_K$ disjoint from $S.$ 

If $K_i$ are a sequence of fibered genus $g$ knots in $S^3$ with distinct Alexander polynomials, which by \cite{Kanenobu81} always exist if $g>1,$ then the $E(n)_{K_i}$'s are all mutually non-diffeomorphic as they are distinguished by their Seiberg-Witten invariants, see \cite{FintushelStern98}. Thus by Lemma~\ref{lem:extend} we conclude that the $S_{n,K_i}$ are all mutually non-diffeomorphic Stein fillings of $\xi_{n-1+2g,1}.$ It follows easily from \cite{Boyer86} that these manifolds are all homeomorphic. 

The existence of infinitely many such examples of contact 3-manifolds, as asserted in the theorem, now follows by taking fiber sums of the $E(n)_{K_i}$'s.
\end{proof}

\def\cprime{$'$} \def\cprime{$'$}

\end{document}